\documentclass[11pt,english]{amsart}
\usepackage{fixltx2e}
\usepackage{amsmath,amstext,amsbsy,amsthm,flafter,amssymb}
\usepackage{leftidx}
\usepackage{amsrefs}
\usepackage{tikz}
\usetikzlibrary{shapes,arrows,calc}
\usepackage{pgfplots}
\usepackage{amsaddr}
\usepackage{hyperref}

\newcommand{\rrangle}{\rangle\!\rangle}
\newcommand{\nequiv}{\not\equiv}

\theoremstyle{plain}
\newtheorem{lemma}{Lemma}[section]
\newtheorem{proposition}[lemma]{Proposition}
\newtheorem{theorem}[lemma]{Theorem}

\newcommand{\C}{\ensuremath{\mathbb{C}}}
\newcommand{\Z}{\ensuremath{\mathbb{Z}}}
\newcommand{\N}{\ensuremath{\mathbb{N}}}

\newcommand{\A}{\ensuremath{\mathcal{A}}}

\newcommand{\Hi}{\ensuremath{\mathcal{H}}}

\newcommand{\U}{\ensuremath{\mathcal{U}}}
\newcommand{\Kq}{\ensuremath{\mathcal{K}_q}}

\newcommand{\sut}{\ensuremath{\mathfrak{su}(2)}}
\newcommand{\jun}{\ensuremath{|j,\mu,n\rrangle}}
\newcommand{\half}{\ensuremath{\frac{1}{2}}}

\newcommand{\zpa}{\ensuremath{\upsilon}}
\newcommand{\cop}{\Delta}
\newcommand{\ts}{\otimes}
\newcommand{\id}{\hbox{id}}

\newcommand{\can}{\mathtt{can}}

\let\le\left
\let\r\right

\begin{document}
\title[Quantum lens spaces]{The geometry of quantum lens spaces: real spectral triples and bundle structure}
\author[A.~Sitarz]{Andrzej Sitarz\textsuperscript{1}}%
\address{Jagiellonian University, Institute of Physics, \\
ul. prof. Stanis\l{}awa \L{}ojasiewicza 11, 30-348 Krak\'ow, Poland,}%
\address{Institute of Mathematics of the Polish Academy of Sciences, \\ \'Sniadeckich 8, Warszawa, 00-950 Poland.}%
\email{andrzej.sitarz@uj.edu.pl}%
\thanks{\textsuperscript{1} Partially supported by NCN grant 2011/01/B/ST1/06474.}%
\author[J.J.~Venselaar]{Jan Jitse Venselaar}%
\address{Mathematisches Institut, Georg-August Universit\"at G\"ottingen, Bunsenstra\ss e 3-5, D-37073 G\"ottingen, Deutschland}
\email{jvensel@uni-math.gwdg.de}

\begin{abstract}We study almost real spectral triples on quantum lens spaces, as orbit spaces of free actions of cyclic groups on the spectral geometry on the quantum group $SU_q(2)$. These spectral triples are given by weakening some of the conditions of a real spectral triple. We classify the irreducible almost real spectral triples on quantum lens spaces and we study unitary equivalences of such quantum lens spaces. Applying a useful characterization of principal $U(1)$-fibrations in noncommutative geometry, we show that all such quantum lens spaces are principal $U(1)$-fibrations over quantum teardrops.\end{abstract}

\maketitle
Lens spaces, orbit spaces of free actions of cyclic groups on odd-dimensional spheres, were first introduced in 1884 by Walther Dyck~\cite{british1885report}. Lens spaces are interesting because they are some of the simplest manifolds exhibiting the difference between homotopy type and homeomorphism type. Quantum lens spaces were introduced in~\cite{MR2015735}. As $C^*$-algebras, they are isomorphic to graph-$C^*$-algebras of certain finite graphs.

In this article we study almost real spectral triples on \emph{quantum lens spaces}, as orbit spaces of free actions of cyclic groups on the spectral geometry on the quantum group $SU_q(2)$ of Woronowicz~\cite{MR890482}, as constructed in~\cite{dbrowski_dirac_2005}. These spectral triples are given by weakening some of the conditions of a real spectral triple, just like in~\cite{dbrowski_dirac_2005}. We classify the irreducible spectral geometries on such quantum lens spaces and we study unitary equivalences of such quantum lens spaces. We also derive a way of computing the Dirac spectrum of all these lens spaces, generalizing a result of~\cite{bar_dirac_1992}.

Finally, we study the structure of these quantum lens spaces as fiber bundles over so-called  quantum teardrops, complementing  the work of~\cite{MR2989456}. We show that in the noncommutative setting, all such quantum lens spaces are principal $U(1)$-bundles over a quantum teardrop.

\section{The equivariant spectral triple on $SU_{q}(2)$}\label{sec:spectral triple suq2}

We recall the construction of the equivariant real spectral triple on $SU_q(2)$ from~\cite{dbrowski_dirac_2005}. This is not a real spectral triple in the sense of~\cite{connes_noncommutative_1995}, since the opposite algebra only commutes with the algebra up to compact operators. This was done in order to cope with certain ``no go-theorems'', which showed that it was impossible for a $U_q(\mathfrak{su}(2))$-equivariant spectral to satisfy all conditions of a real spectral triple~\cite[Remark 6.6]{dbrowski_dirac_2005}.

Let $q$ denote a real number, $0\leq q<1$. Let $\A(SU_q(2))$ be the $*$-algebra generated by the two elements  $a$, $b$, satisfying the following relations:
\begin{subequations}\label{eqn:commutation relations suq2}
\begin{align}
 ba &= q ab\\
 b^*a &= q ab^*\\
 b b^* &= b^* b\\
 a^*a + q^2 b^* b &=1\\
 a a^* + b b^* &= 1.
\end{align}
\end{subequations}
From these relations it follows that $a^* b = q ba^*$, $a^* b^* = q b^* a^*$ and $[a, a^*] = (q^2 - 1) b b^*$. If $q=1$, we recover the generators of $SU(2)$ as a commutative space.

There is a vector space basis $e_{k l m}$ of $\A(SU_q(2))$, given by monomials of the form
\begin{equation}e_{k l m}: = \begin{cases} a^k b^l {b^*}^m & k\in\Z, k\geq 0, l,m\in \N\\
b^l {b^*}^m {a^*}^{-k} & k\in \Z, k<0,l,m\in \N
\end{cases}
\label{eqn:suq2 space algebra basis}
\end{equation}

The Hilbert space of the spectral triple can be written as a direct sum
\[ \Hi = \Hi^\uparrow \oplus \Hi^\downarrow,\]
with $\Hi^\uparrow$ and $\Hi^\downarrow$ spanned by an orthonormal basis $ | j \mu n \uparrow \rangle$ and $|j\mu n\downarrow \rangle$ respectively, such that $j=0,\half,1,\frac{3}{2},\ldots$, $\mu = -j, -j+1,\ldots,j$, and $n=-j-\half, -j + \half,\ldots, j+\half$ for the $\uparrow$ part, and $n=-j+\half, -j + \half,\ldots, j-\half$ for the $\downarrow$ part. As a convenient shorthand, we will write 

\begin{equation} \jun := \begin{pmatrix} |j\mu n \uparrow\rangle\\ |j\mu n \downarrow \rangle\end{pmatrix}.\end{equation}

An equivariant representation of the algebra $\A(SU_q(2))$ on this Hilbert is given by:
\begin{proposition}[\cite{dbrowski_dirac_2005}]\label{prop:algebra action}
The following representation $\Pi$ of $\A(SU_q(2))$ on the Hilbert space $\Hi$ with orthonormal basis $|j \mu n\rrangle$, is equivariant with respect to the action of $\U_q(\sut)$ defined in~\cite{dbrowski_dirac_2005}.
\begin{subequations}\label{eqn:suq2 algebra action}
 \begin{align}
\Pi\le(a\r) \jun &= \alpha_{j\mu n}^+ | j^+ \mu^+ n^+\rrangle + \alpha_{j \mu n}^- | j^- \mu^+ n^+\rrangle\\
\Pi\le(b\r) \jun &= \beta_{j\mu n}^+ | j^+ \mu^+ n^-\rrangle + \beta_{j \mu n}^- | j^- \mu^+ n^-\rrangle\\
\Pi\le(a^*\r) \jun &= \tilde{\alpha}_{j\mu n}^+ | j^+ \mu^- n^-\rrangle + \tilde{\alpha}_{j \mu n}^- | j^- \mu^- n^-\rrangle\\
\Pi\le(b^*\r) \jun &= \tilde{\beta}_{j\mu n}^+ | j^+ \mu^- n^+\rrangle + \tilde{\beta}_{j \mu n}^- | j^- \mu^- n^+\rrangle
\end{align}
\end{subequations}
where $\alpha^\pm_{j\mu n}$,$\beta^\pm_{j\mu n}$, $\tilde{\alpha}_{j \mu n}^\pm$ and $\tilde{\beta}_{j \mu n}^\pm$  are bounded triangular $2\times 2$ matrices. For our purposes, we only need to know that they are not diagonal with respect to the $\uparrow$ and $\downarrow$ decomposition of the Hilbert space, i.e.\ the algebra maps from the $\uparrow$ to the $\downarrow$ and vice versa. 
\end{proposition}

The Dirac operator of the spectral triple constructed in~\cite[Section 5]{dbrowski_dirac_2005} is given by:
\begin{align}
 D\begin{pmatrix} |j\mu n\uparrow\rangle\\
    | j'\mu'n'\downarrow\rangle
  \end{pmatrix} &= \begin{pmatrix} \le(2j+\frac{3}{2}\r) |j\mu n\uparrow\rangle\\ -\le(2j'+\half\r)|j'\mu'n'\downarrow\rangle\end{pmatrix}.\label{eqn:Dirac operator suq2}\\
\intertext{The reality operator, constructed in~\cite[Section 6]{dbrowski_dirac_2005} is given by:}
 J\begin{pmatrix} |j\mu n\uparrow\rangle\\
    | j'\mu'n'\downarrow\rangle
  \end{pmatrix} &= \begin{pmatrix} i^{2(2j+\mu+n)} |j, -\mu, -n\uparrow\rangle\\  i^{2(2j'-\mu'-n')}|j',-\mu',-n'\downarrow\rangle\end{pmatrix}.\label{eqn:reality operator suq2}
\end{align}

Together, the algebra $\A(SU_q(2))$, with representation as given in Proposition~\ref{prop:algebra action}, Dirac operator $D$ given by \eqref{eqn:Dirac operator suq2}, and reality operator $J$ given by \eqref{eqn:reality operator suq2}, satisfy most conditions as stated in~\cite{connes_noncommutative_1995}, with some slight modifications. These modifications are that $[\Pi(x),J\Pi(y^*)J^{\dagger}]$ and $[[D,\Pi(x)],J\Pi(y^*)J^{\dagger}]$ are not exactly $0$, but lie in the two-sided ideal $\Kq$ in $B(\Hi)$ generated by the compact, positive, trace class operators
\begin{equation}
 M_q : M_q \jun = q^j \jun. \label{eqn:compact trace lens}
\end{equation}
This means that the first order and reality conditions should be modified appropriately.

Also, it is currently unknown whether the Hochschild cycle condition and the Poincar\'e duality condition are satisfied. 

\section{Topological quantum lens spaces}\label{sec:topological lens spaces}
Commutative lens spaces are defined as the quotient of an odd-dimensional sphere by a free action of a finite cyclic group.
Similarly, a $3$-dimensional quantum lens space can be defined as the invariant algebra of $SU_q(2)$ under an action of $\Z/p\Z$:
\begin{subequations}
\begin{align}
 g \triangleright a &= e^{\frac{2 \pi i}{p} r_1} a\label{eqn:g action on a}\\
 g \triangleright b &= e^{\frac{2 \pi i}{p} r_2} b,\label{eqn:g action on b}
\end{align}
\end{subequations}
with $r_1$ and $r_2$ integers. We denote the invariant algebra as $L_q(p;r_1,r_2)$. It is easy to see that $L_q(p;r_1,r_2)$ is spanned as a vector space by basis elements of the form $e_{klm}$ with $r_1 k + r_2(l-m) \equiv 0 \bmod p$ in the notation of \eqref{eqn:suq2 space algebra basis}.

Let $\epsilon(r) = (r_1+r_2) \bmod 2$. We have:
\begin{proposition}\label{prop:action of Z_p}
Choose a generator $g$ of $\Z/p\Z$. The action defined by $\A(SU_q(2))$ by \eqref{eqn:g action on a} and \eqref{eqn:g action on b} together with the unitary representation $\zpa$ of $\Z/p\Z$ on $\Hi$ given by:
\begin{equation}
 \zpa(g) \jun = e^{\frac{2 \pi i}{p} \le((r_1+r_2)\mu + (r_1-r_2) n + \half \epsilon(r)\r)} \jun,
\end{equation}
makes $\Hi$ into a left $\Z/p\Z$-comodule over $\A(SU_q(2))$.
\end{proposition}
The presence of $\epsilon(r)$ guarantees that the expression in brackets is an integer.
\begin{proof}
Since $\Z/p\Z$ is a group, we demand that the action of $\A$ on $\Hi$ is equivariant with respect to the group viewed as a Hopf algebra, with coproduct $\Delta(h) = h\ts h$ for all $h\in \Z/p\Z$. From this it follows that $\zpa(g) |j\mu n\rrangle = c_{j\mu n} |j\mu n\rrangle$, for some $c_{j\mu n}\in \C$, since otherwise this would not be compatible with the equivariance condition $\zpa(g) \Pi(a) v = \Pi(g_{(1)} a) \zpa(g_{(2)}) v$, where $g_{(1)} a$ is $e^{\frac{2 \pi i r_1}{p}} a$, as defined in \eqref{eqn:g action on a}
We can then calculate:
\begin{align*}
  \zpa(g) \Pi(a) | j\mu n\rrangle &= \Pi\le(e^{\frac{2\pi i}{p}r_1}a\r) \zpa(g) | j\mu n\rrangle,\\
  \zpa(g) \Pi(b) | j\mu n\rrangle &= \Pi\le(e^{\frac{2\pi i}{p}r_2}b\r) \zpa(g) | j\mu n\rrangle.
\end{align*}
From~\eqref{eqn:suq2 algebra action} we see that the action of $a$ and $a^*$ on $\Hi$ leaves the difference $\mu-n$ constant, and $b$ leaves the sum $\mu+n$ constant.

We get the recurrence relations $\zpa(g)|j^{\pm}\mu^+ n^+\rrangle = e^{\frac{2\pi i}{p}r_1} \zpa(g) |j \mu n\rrangle$ and $\zpa(g)|j^{\pm}\mu^+ n^-\rrangle = e^{\frac{2\pi i}{p}r_2} \zpa(g) |j \mu n\rrangle$. Solving these recurrence relations, we see that $\zpa(g) |j\mu n\rrangle = e^{\frac{2\pi i}{p} \le( (\mu+n)r_1 + (\mu-n)r_2\r) + c}$, where $c$ is any constant. In order to have $g^p=1$, we see that $c = \epsilon(r)$ plus an additional integer, which can be set to zero.
\end{proof}

If $r_1$ is coprime to $p$, we see that $L_q(p;r_1, r_2) = L_q(p;1,r_1^{-1} r_2)$ by multiplying each component by $e^{\frac{2\pi i}{p} r_1^{-1}}$ (where $r_1^{-1}$ is meant $\bmod\ p$). If $r_1$ is coprime to $p$, we will write $L_q(p,r)$ for the lens space $L_q(p;r_1,r_2)$, with $r =r_1^{-1} r_2\bmod p$.

\begin{proposition}\label{prop:freeness of coaction}
If $r = r_1^{-1} r_2 \bmod p$ is coprime to $p$, the above defined action on $SU_q(2)$ extends to its $C^*$-algebra and is free in the sense of Ellwood~\cite{Ellwood:2000}.
\end{proposition}
Strictly speaking, what we prove here is freeness of \emph{coaction}, not of the action, but it is clear that in the simple case of $\Z/p\Z$ there is a simple $1-1$ mapping between actions and coactions.
\begin{proof}
First of all, we can easily translate the action of $\Z/p\Z$ to the right 
coaction of $C(\Z/p\Z)$ on the $SU_q(2)$ algebra:
\[ \cop_R x = \sum_{h \in \Z/p\Z} (h \triangleright x) \ts \delta_h, \]
where $\delta_h$ is the function defined by $\delta_h(h) = 1$ and $0$ on other group elements, and extended by linearity.
Recall that that the freeness of a coaction of a Hopf algebra $H$ on a $C^*$-algebra 
$A$ means (for a right coaction) that the spans of $ (A \ts \id) \cop_R (A) $  and 
$ \cop_R(A) (A \ts \id)$ are dense in $A \ts H$ for the minimal tensor product.

For the action, which defines lens spaces the freeness is easy to verify. Consider the 
identity:
\[ {(a a^* + b b^*)}^r = 1,\]
which can be rewritten, using the commutation relations as:
\[ a^r {(a^*)}^r + b P(a,a^*,b,b^*) = 1, \]
where $P$ is some polynomial in the generators. Therefore,
\[ \cop_R(a^r) \left( {(a^*)}^r \ts 1 \right) + \cop_R(b) \left( P(a,a^*,b,b^*) \ts 1 \right) =
    \sum_{k=0}^{p-1} e^{2\pi i \frac{rk}{p}} \ts \delta_k = 1 \ts f, \]
where $f$ is the function on $\Z/p\Z$:
\[ f(k) =  e^{2\pi i \frac{rk}{p}}. \]
As for $r>0$ and $p$ relatively prime the function $f$ generates the algebra $C(\Z/p\Z)$ 
this finishes the proof. 
\end{proof}
In the remainder we will assume that $r$ and $p$ are coprime.

Observe that if we replace $r$ by $r-p$ the action on the generators does not change. If we take $p$ even then $r$ is necessarily odd, and $\epsilon(r)$ is $0$. If $r$ is even, then $p$ is necessarily odd, and $r-p$ is odd, and on the Hilbert space the actions determined by $(p,r)$ and $(p,r-p)$ are equivalent since:
\begin{align*}
 e^{\frac{2 \pi i}{p} \le(( 1+ (r-p))\mu + (1-(r-p))n\r)} &= \le( e^{\frac{2 \pi i p (n-\mu)}{p}} e^{-\frac{\pi i \epsilon(r)}{p}}\r) e^{\frac{2\pi i}{p} \le((1+r)\mu + (1-r)n + \half \epsilon(r) \r)}\\
&= \le( - e^{-\frac{\pi i \epsilon(r)}{p}}\r) \rho(g),
\end{align*}
where we have used that $n-\mu \in \half\Z\backslash\Z$. Since
\begin{equation}
 {\le(-e^{-\frac{\pi i \epsilon(r)}{p}}\r)}^p = 1,
\end{equation}
for odd $p$, the actions are equivalent.
For this reason, if $r$ is coprime to $p$, we can always take $r$ odd, and $\epsilon(r) = 0$.
Let $K= 0, 1, 2,\ldots, p-1$. We define $\Hi_K$ as the eigensubspace of $\Hi$ for the action of $g$ with eigenvalue $e^{\frac{2 \pi i K}{p}}$. Further, let $L_q(p,r) \simeq L_q(p;1,r)$ be the subalgebra of $\A(SU_q(2))$ which is invariant under the action of $\Z/p\Z$. By construction, we have $x\in L_q(p,r)$, $\Pi(x) \Hi_K \subset \Hi_K$. It is also obvious that $D \Hi_K \subset \Hi_K$.

\begin{proposition}\label{prop:lens space real structure} The equivariant real structure $J$, as given in~\eqref{eqn:reality operator suq2} satisfies:
\begin{equation} J \Hi_K = \Hi_{K'},\end{equation}
where $K+K' \equiv 0 \bmod p$.
\end{proposition}
\begin{proof}
We have $J|j\mu n\rrangle = c_{j\mu n} |j-\mu -n \rrangle$ with $c_{j\mu n}$ the complex number defined in \eqref{eqn:reality operator suq2}, and from \[(1+r)\mu + (1-r) n = K \bmod p,\] it follows that \[(1+r)\cdot (-\mu) + (1-r)\cdot (-n) = -K \bmod p.\qedhere\]
\end{proof}

\section{Geometry of quantum lens spaces}\label{sec:geometry of quantum lens}
We now turn to the geometrical properties of the almost real spectral triple of the quantum lens space. As stated at the end of Section~\ref{sec:spectral triple suq2}, we modify some conditions of a real spectral triple, exactly as in~\cite{dbrowski_dirac_2005}, i.e.\ the real structure and the first order condition only hold up to the compact operators of positive trace class defined in~\eqref{eqn:compact trace lens}.

Furthermore, it is unknown if the finiteness condition, the Hochschild cycle condition and the Poincar\'e duality are satisfied for $SU_q(2)$. This means that also for $L_q(p,r)$ we do not know if they are satisfied.

We call a structure, satisfying all conditions of~\cite{connes_noncommutative_1995}, with the modification of the first order condition, and the removal of the finiteness condition, the Hochschild cycle condition and Poincar\'e duality an \emph{almost real spectral triple}.

\begin{proposition} Let $L_q(p,r)$, $q\in(0,1)$, be the quantum lens space as defined above. Then for any $K=0,1,\ldots p-1$, the Hilbert space $\Hi_K \oplus \Hi_{K'}$, where $K+K' \equiv 0 \bmod p$, the reality structure $J$ and the Dirac operator $D$ taken as the restrictions of $J$ and $D$ from the $\A(SU_q(2))$ almost real structure constitute an almost real spectral triple over the quantum lens space $L_q(p,r)$.\end{proposition}
\begin{proof}
Almost all the usual conditions (KO-homology class, regularity) for a weakly real spectral triple are easily seen to carry over from the $SU_q(2)$ case. The only slightly non-trivial conditions are the compact resolvent condition, the metric dimension and the finiteness condition. 
 
The compact resolvent condition and metric dimension follow from the fact that for $j$ big enough, there always exist $\mu$ and $n$ such that there are vectors $\jun \in \Hi_K$, and thus also in $\Hi_{K'}$, hence the dimension growth is satisfied. The compact resolvent condition also follows from this, and the fact that the dimension of the kernel of $D$ is finite dimensional.
\end{proof}

The finiteness condition is the statement that the of smooth vectors $\Hi_K^\infty := \bigcap_{k=1}^\infty \text{Dom}D^k$ is a finitely generated projective module over the smooth algebra. The set of smooth vectors of $SU_q(2)$, and thus of $L_q(p,r)$ is the classical one, however the precise algebra of smooth elements of the algebra is as of yet unknown, and thus we can at the moment not say something about the finiteness. We have the following partial result:
\begin{lemma}\label{lem:finitely generated line bundles}
The subspaces $\A_K\subset \A(SU_q(2))$, $K\in \Z/p\Z$, $\A_K := \{a\in \A(SU_q(2)): \rho(g) a = e^{\frac{2\pi i}{p} K} a\}$, are finitely generated projective modules over $L_q(p,r)$.
\end{lemma}
\begin{proof}
It is clearly enough to find a finite set $\{\xi_i\}_{i=1}^n \in \A_K$ such $\sum_{i=1}^n \xi_i^* \xi_i = 1$. This then implies that the $\xi_i$ generate $\A_K$ as a left module, and projection $p\in M_n(\A_0)$ given component-wise by $p_{ij}= \xi_i \xi_j^*$. Also, it is enough to check that such a set exists for $\A_1$, because then a set for $\A_K$ is then given by taking all possible $K$-fold products of $\xi_i$ for $\A_1$.

Such a set for $\A_1$ can be constructed as follows.
We start with $a^* a + q^2 b^* b (=1)$. It is clear that $a\in \A_1$. If $q=0$, this is enough. Assume $q\neq 0$. If $r=1$, $b\in \A_1$ and we are also done.
If $r>1$, $b\notin \A_1$. In order to construct elements in $\A_1$, we use repeated multiplication by suitable decompositions of $1$, for example $b^* b = b^* a a^* b + b^*b b^* b$, etc.

From an element with decomposition in $\A_k$, this will construct two elements, one with decomposition in $\A_{k-1}$ (adding an $aa^*$-term), and one with decomposition in $\A_{k-r}$ (adding a $bb^*$-term). One can keep iterating this process for each element with a decomposition not in $\A_1$. We see that the minimal exponent of the $b^*$ part of elements with no decomposition in $\A_1$ is an increasing function. This is because after at most $k-1$ steps, adding just $aa^*$ will give us an element in $\A_1$ from an element in $\A$. However, if the exponent of $b^*$ is big enough, one can rearrange the $b$ and $b^*$ in such a way, picking up some powers of $q$ from the non-trivial commutation relations of $a$ and $b$, that one can find a decomposition in $\A_1$. 

Switching a $b$ with a $b^*$ will make an element with decomposition in $\A_k$ into an element with decomposition $\A_{k+2r}$. If $p$ is odd, and the exponent of $b^*$ is high enough in the element, one can chose a decomposition in $\A_1$, because $2r$ and $p$ are coprime. If $p$ is even this only happens when $k$ is odd, but this can always be arranged, by observing that the elements are of the form $a^j {(b^*)}^{n-j} b$ after $n$ steps, which lies in $\A_k$, $k=r(-1+n-j) -j$. We know $r$ is odd, so if $n$ is even this is always odd.
\end{proof}

\subsection{Irreducibility}\label{ssec:irred}

Though any of the above listed spectral geometries for the quantum lens spaces is admissible from the point of view of the noncommutative axiomatic approach, not all correspond to spin structures on commutative lens spaces. We demand that our spectral triple be irreducible as the analogue of a connected manifold in commutative geometry. 
If we use~\cite[Definition 11.2]{bonda_elements_2001} or~\cite[Definition 2.1]{iochum_classification_2004}, all 
real spectral triples above are irreducible, since the $J$ operator interchanges the $\Hi_K$ and $\Hi_{K'}$ spaces. If we use the definition of~\cite[Remark 6 on p.163]{connes_gravity_1996}, only irreducibility with respect to the algebra action and Dirac operator are demanded. None of the above real spectral triples are then irreducible, however the cases where $K=K'=0$ and $K=K'=p/2$ if $p$ is even can be made irreducible by dropping one of the two copies of the Hilbert space, and setting the real spectral triple to be $(L_q(p,r),\Hi_K,D,J)$. Since $J\Hi_K \subset \Hi_K$ in this case, this is a well-defined spectral triple, and irreducible. Since the action of the Dirac operator is diagonal with respect to the $\uparrow$ and $\downarrow$ decomposition of the Hilbert space, it is crucial for this to work that the algebra action is not diagonal.

For even $p$ we obtain two possible spin structures, for odd $p$ just one, just as in the commutative case~\cite{franc_spin_1987}.
\begin{theorem}\label{thm:lens spaces}
The quantum lens space $L_q(p,r)$ admits one irreducible almost real spectral triple coming from the spectral triple on $SU_q(2)$ if $p$ is odd, and two if $p$ is even.
The spectral geometries are given in the two cases by
 \begin{itemize}
  \item $p \bmod 2 = 1$: $(L_q(p,r),\Hi_0,D,J)$.
  \item $p = 2P$: $(L_q(p,r),\Hi_0,D,J)$ and $(L_q(p,r),\Hi_P,D,J)$.
 \end{itemize}
with $D$ the Dirac operator as described in~\eqref{eqn:Dirac operator suq2} and $J$ the operator given in~\eqref{eqn:reality operator suq2}.
\end{theorem}

\subsection{The spectrum of the Dirac operator}\label{subsec:lens space spectrum}
Let us recall that the spectrum of the Dirac operator over $\A(SU_q(2))$ (with appropriate normalization) is given by:
\begin{align*}
 D | j,\mu,n,\uparrow\rangle &= \le(2 j + \frac{3}{2}\r) | j,\mu,n,\uparrow\rangle & \text{with } j=0,\half,\ldots\\
&\text{with multiplicity: } (2j+1)(2j+2)\\
D | j,\mu,n,\downarrow\rangle &= -\le(2j +\frac{1}{2}\r) | j,\mu,n,\downarrow\rangle & \text{with } j=\half,1,\ldots\\
&\text{with multiplicity: } (2j +1) 2j
\end{align*}
Note that the spectrum is symmetric. Since in the construction of the spectral geometries on quantum lens spaces we keep the Dirac operator of $SU_q(2)$ (and only restrict the Hilbert space), the spectrum remains unchanged, only the multiplicities differ. 

By construction the spectral triple thus constructed has a Dirac spectrum independent of $q$, and for us to compute this spectrum it would be enough to refer to a computation of the Dirac spectrum as computed in the commutative case, for example~\cite[Theorem 5]{bar_dirac_1992} and~\cite{MR3108692}. However, there the spectrum is only explicitly computed for the $L_q(p,p-1)$ in the notation used in this article, so we give a formula with which to compute the Dirac spectrum of a (quantum) lens space for all values of $p,r$ coprime, for all spin structures. 

With the explicit description of the weakly real spectral triple $(L_q(p,r),\Hi_K,D)$, we can reduce the problem of computing these multiplicities to a number theoretic problem of solving congruence relations.

\begin{proposition}\label{prop:spectrum on lens space}
The eigenvalues of $D$ belonging to an irreducible almost real spectral triple as in Theorem~\ref{thm:lens spaces} with Hilbert space $\Hi_K$ ($K=0$ or $K=\half p$) are $2j + \frac{3}{2}$ and $-2j - \half$, with respective multiplicity $N^+(j)$ and $N^-(j)$, (either of them could be $0$, which means that the value is not present in the spectrum) where $N^{\pm}(j)$ denotes the number of solutions to the equation:
\begin{equation} (1+r) \mu + (1-r) n \equiv K \bmod p\label{eqn:values for mu and n}\end{equation}
with $-j\leq \mu \leq j$ and $-(j\pm\half) \leq n \leq j \pm \half$.
\end{proposition}
The exact calculation of the number of eigenvalues is a tedious task. The solution depends heavily on the properties of $(1+r)$ and $(1-r)$, in particular on the greatest common divisor of $1\pm r$ and $p$. Although in each case the explicit solutions for $\mu$ and $n$ can be easily found, calculating the number of solutions for a given $j$ is rather difficult for an abstract choice of $r$ and $p$.

 \tikzstyle{nonvertex}=[circle,fill=black!10,minimum size=2pt,inner sep=0pt]
 \tikzstyle{vertex}=[circle,fill=black!35,minimum size=4pt,inner sep=0pt]
 \tikzstyle{lensvertexint}=[star,fill=black,minimum size=6pt,inner sep=0pt]
 \tikzstyle{lensvertexhalf}=[diamond,fill=black,minimum size=8pt,inner sep=0pt]

%
\newcommand{\drawnode}[2]{
  \pgfmathparse{int(mod(#1+#2-1,2))}
	  \ifnum\pgfmathresult=0{
	    \pgfmathparse{int(mod((1+\r)*#1/2+(1-\r)*#2/2,\p))}
	    \ifnum\pgfmathresult=0{
	      \pgfmathparse{int(mod(#2,2))}
	      \ifnum\pgfmathresult=0{\node[lensvertexint] at (#1/2,#2/2) {};}
	      \else{\node[lensvertexhalf] at (#1/2,#2/2) {};}\fi}
	    \else{\node[vertex] at (#1/2,#2/2) {};}\fi}
	  \else{\node[nonvertex] at (#1/2,#2/2) {};}\fi
}%

\newcommand{\drawnodep}[2]{
  \pgfmathparse{int(mod(#1+#2-1,2))}
	  \ifnum\pgfmathresult=0{
	    \pgfmathparse{int(abs(mod((1+\r)*#1/2+(1-\r)*#2/2,\p)))}
	    \ifnum\pgfmathresult=1{
	      \pgfmathparse{int(mod(#2,2))}
	      \ifnum\pgfmathresult=0{\node[lensvertexint] at (#1/2,#2/2) {};}
	      \else{\node[lensvertexhalf] at (#1/2,#2/2) {};}\fi}
	    \else{\node[vertex] at (#1/2,#2/2) {};}\fi}
	  \else{\node[nonvertex] at (#1/2,#2/2) {};}\fi
}%
\newcommand{\drawj}[1]{
  \draw[style=help lines,dashed] (-#1-1/2,-#1) rectangle (#1+1/2,#1);
}%

For illustration, we show here some pictures of what the spectra for small $p$ looks like. We represent the basis vectors $|j\mu n \uparrow\rangle,|j\mu n \downarrow\rangle$ of the Hilbert space $\Hi$ of the spectral triple on $SU_q(2)$ as defined in Proposition~\ref{prop:algebra action} as a lattice, with $\mu$ on the horizontal axis and $n$ on the vertical axis. We project the $j$ coordinate away, since it does not play a role for determining whether a vector lies in $\Hi_K$, only in confining the possible $\mu$ and $n$. We illustrate this by drawing lines through the allowed values (in the form of rectangles around the origin) for $j=3$ and in the $\uparrow$ part of $\Hi$. We draw a circle \begin{tikzpicture}[scale=0.7] \node[shape=circle,draw] {};\end{tikzpicture} through the origin $\mu=0,n=0$.

The stars (\begin{tikzpicture}\node[lensvertexint] {};\end{tikzpicture}) and diamonds(\begin{tikzpicture}\node[lensvertexhalf] {};\end{tikzpicture}) represent basis vectors of the $\Hi_K$ Hilbert space of the lens space $(L_q(p,r),\Hi_K,D,J)$. The circles (\begin{tikzpicture}\node[vertex] {};\end{tikzpicture}) represent basis vectors of the Hilbert space of $SU_q(2)$ which are not part of the lens space. The stars are the allowed values for integer valued $j$, the diamonds are the allowed values for half-integer valued $j$.
\begin{figure}[!htb]
\begin{minipage}[b]{0.49\textwidth}
\centering
 \begin{tikzpicture}[scale=0.7]
 \def\p{2}
 \def\r{1}
 \foreach \n in {-7,-6,...,7}
   \foreach \m in {-6,-5,...,6}
      {\drawnode{\n}{\m}}
\node[shape=circle,draw] at (0,0) {};
\foreach \j in {0,1,...,3}
  {\drawj{\j}}
\end{tikzpicture}
\caption{$\Hi_0$ for $p=2,r=1$}
\end{minipage}
\begin{minipage}[b]{0.49\textwidth}
  \begin{tikzpicture}[scale=0.7]
 \def\p{2}
 \def\r{1}
 \foreach \n in {-7,-6,...,7}
   \foreach \m in {-6,-5,...,6}
      {\drawnodep{\n}{\m}}
\node[shape=circle,draw] at (0,0) {};
\foreach \j in {0,1,...,3}
  {\drawj{\j}}
\end{tikzpicture}
\caption{$\Hi_1$ for $p=2,r=1$}
\end{minipage}
\end{figure}\\
\begin{figure}[!htb]
\begin{minipage}[b]{0.49\textwidth}
 \begin{tikzpicture}[scale=0.7]
 \def\p{5}
 \def\r{1}
 \foreach \n in {-7,-6,...,7}
   \foreach \m in {-6,-5,...,6}
      {\drawnode{\n}{\m}}
\node[shape=circle,draw] at (0,0) {};
\foreach \j in {0,1,...,3}
  {\drawj{\j}}
\end{tikzpicture}
\caption{$\Hi_0$ for $p=5,r=1$}\label{fig:lattice p=5 r=1}
\end{minipage}
\begin{minipage}[b]{0.49\textwidth}
 \begin{tikzpicture}[scale=0.7]
 \def\p{5}
 \def\r{-3}
 \foreach \n in {-7,-6,...,7}
   \foreach \m in {-6,-5,...,6}
      {\drawnode{\n}{\m}}
\node[shape=circle,draw] at (0,0) {};
\foreach \j in {0,1,...,3}
  {\drawj{\j}}
\end{tikzpicture}
\caption{$\Hi_0$ for $p=5,r=-3$}\label{fig:lattice p=5 r=-3}
\end{minipage}
\end{figure}
\begin{figure}[!htb]
\begin{minipage}[b]{0.49\textwidth}
 \begin{tikzpicture}[scale=0.6]
 \def\p{7}
 \def\r{-5}
 \foreach \n in {-8,-7,...,8}
   \foreach \m in {-8,-7,...,8}
      {\drawnode{\n}{\m}}
\node[shape=circle,draw] at (0,0) {};
\foreach \j in {0,1,...,3}
  {\drawj{\j}}
\end{tikzpicture}
\caption{$\Hi_0$ for $p=7,r=-5$}\label{fig:lattice p=7 r=-5}
\end{minipage}
\begin{minipage}[b]{0.49\textwidth}
 \begin{tikzpicture}[scale=0.6]
 \def\p{7}
 \def\r{3}
 \foreach \n in {-8,-7,...,8}
   \foreach \m in {-8,-7,...,8}
      {\drawnode{\n}{\m}}
\node[shape=circle,draw] at (0,0) {};
\foreach \j in {0,1,...,3}
  {\drawj{\j}}
\end{tikzpicture}
\caption{$\Hi_0$ for $p=7,r=3$}\label{fig:lattice p=7 r=3}
\end{minipage}
\end{figure}

\section{Unitary equivalences}

It is known~\cite{reidemeister_homotopieringe_1935} that two commutative lens spaces $L(p,r)$ and $L(p',r')$ are homeomorphic if and only if $p=p'$ and $r' r \equiv \pm 1 \bmod p$, or $r' \pm r \equiv 0 \bmod p$. 

It is also known that two 3-dimensional lens spaces are homeomorphic if and only if they are (Laplace) isospectral, see~\cite{ikeda_spectra_1979}.

That these concepts are related for lens spaces can be intuitively understood by looking at diagrams as in Section~\ref{subsec:lens space spectrum}. There we see that the diagram of $L(p,r)$ is the same as the diagram of $L(p,r')$, $r'=\pm r^{\pm1}$, up to rotation, mirroring and interchanging the stars and diamonds. For example, $3 \equiv -{(-5)}^{-1} \bmod 7$, and we see in Figures~\ref{fig:lattice p=7 r=-5} and~\ref{fig:lattice p=7 r=3} that they are the same if we rotate Figure~\ref{fig:lattice p=7 r=-5} one quarter clockwise and flip the stars and diamonds. The $L(5,1)$ and $L(5,-3)$ lens space diagrams of Figures~\ref{fig:lattice p=5 r=1} and~\ref{fig:lattice p=5 r=-3}  are very different however, and not related by mirroring and rotations by quarter-turns. 

In~\cite{ikeda_spectra_1979}, using results on these type of lattices from~\cite{MR558314}, it is then shown that these types of lattice isomorphism induce an isomorphism of the algebra of smooth functions.

In the noncommutative case when $q\neq 1$, the lens spaces $L_q(p,r)$ and $L_q(p,r')$ are shown to be unitary equivalent when $r=\pm r'$, if we take care of the equivariant representation.

\begin{theorem}\label{thm:unitary equiv lens}
When $q\in(0,1)$, the weakly real spectral triples given by $(L_q(p,r),\allowbreak\Hi_K,\allowbreak D,\allowbreak J)$ and $(L_q(p,r'),\allowbreak\Hi_K,\allowbreak D,\allowbreak J)$ are unitary equivalent if $r'\equiv -r\bmod p$. The unitary equivalence is implemented by the order order-two automorphism $\sigma(a) = a$, $\sigma(b) = -b^*$ of $\A(SU_q(2))$ and the action $U$ on the Hilbert space given by $U| j\mu n \uparrow\rangle = - |j^+ n \mu \downarrow\rangle$ and $U| j\mu n \downarrow\rangle =  |j^- n \mu \uparrow\rangle$, with $u_{j\mu n}$ a complex number of norm $1$.
\end{theorem}
\begin{proof}
To show that this map is a unitary equivalence, we first study the equivalence of Hilbert spaces.
If $v\in \Hi_K$ with $K=0$ or $K=p/2$, we have $(1+r)\mu + (1-r)n \equiv 0 \bmod p$ or $p/2$ respectively.
If we take $r'=-r$, we see that $(1+r')\mu + (1-r')n = (1-r)\mu + (1+r)n$ and we see that if we interchange $\mu$ and $n$, $v\in \Hi_K$ is mapped to a vector $v\in\Hi_K'$ in the $K=0$ or $K=p/2$ subspace of the $\zpa$ action for $r'=-r$.

To define a compatible action of the algebra on $\Hi_K'$, we see from Proposition~\ref{prop:algebra action} that we need to interchange $b$ and $-b^*$.
Now to show that this indeed gives a unitary equivalence on the algebra, we need to show that the $U^{-1}(\Pi'(-b))U = \Pi(b^*)$, and $U^{-1}(\Pi'(a))U = \Pi(a)$. This can be done by an explicit calculation, using the full matrices $\alpha^\pm_{j\mu n}$ etc. of \eqref{eqn:suq2 algebra action}, taking care of the fact that while \eqref{eqn:suq2 algebra action} uses the shorthand $|j\mu n\rrangle$ notation, this is not respected by the unitary map $U$.
\end{proof}

Of course, from~\cite{MR2015735} it is implicit that as graph $C^*$-algebras $L_q(p,r)$ and $L_q(p,r')$ for all $r$ and $r'$ coprime to $p$ are isomorphic. However, it is unclear if this descends to an isomorphism on the level of smooth algebras, and if this then leads to a unitary equivalence.

Observe that even though the algebras are isomorphic to each other it is not obvious that the spectral 
triples are unitarily equivalent. This is because the construction of spectral triples over lens spaces
is based on the restriction of an equivariant spectral triple over the full $\A(SU_q(2))$ algebra. As it
is generally not true that a restriction of given spectral triple to two subalgebras results in unitary 
equivalent spectral triples.\footnote{A trivial example is that of a torus with a nontrivial spin 
structure - its restriction to two different subalgebras of functions over a circle gives two spectral 
triples which correspond two two distinct spin structures over the circle.}

It should be noted that using similar methods as in~\cite{OlczykowskiSitarz:2013} one can show that
any spectral triple over quantum lens spaces is a restriction of a spectral triple over $\A(SU_q(2))$ 
algebra. This does not guarantee, however, that the resulting spectral triple lifted to $\A(SU_q(2))$ 
is equivariant.

For the other type of equivalences in the commutative case, i.e.\ the $r\rightarrow r^{-1}$ case, we do not know if they give rise to unitary equivalences in the $q\neq 1$ case. For isomorphisms of $L_q(p,r)$ coming from the automorphisms of $\A(SU_q(2))$, as classified in~\cite[Proposition 3.1]{MR2242563}, we can show that they do not give rise to isomorphisms between $L_q(p,r)$ and $L_q(p,r^{-1})$. Take for example an element of the form $a^* b^l {b^*}^m \in L_q(p,r)$, with $r(l-m)-1\equiv 0 \bmod p$, i.e.\ $l-m\equiv r^{-1} \bmod p$. We have $r^{-1}(l-m) \equiv {(r^{-1})}^2 \nequiv 1 \bmod p$ if $r^{-1} \nequiv r$. This means that if $r^{-1} \nequiv r$, the identity automorphism is not a map from $L_q(p,r)$ to $L_q(p,r^{-1})$. Also, we have $r^{-1} (m-l) \equiv r^{-1}\cdot (-r^{-1}) \nequiv 1 \bmod p$ if $r^{-1} \nequiv -r$, hence the automorphism $a\rightarrow a$, $b\rightarrow b^*$ is not a map from $L_q(p,r)$ to $L_q(p,r^{-1})$. Hence the automorphisms of $\A(SU_q(2))$ do not give homomorphisms from $L_
q(
p,r)$ to $L_q(p,r^{-1})
$ if $r^{-1} \nequiv \pm r$. The same argument holds for $L_q(p,r)$ and $L_q(p,-r^{-1})$.

\section{Quantum teardrops and principal fiber bundles}
In~\cite{MR2989456}, it was shown that certain quantum lens spaces, namely the lens spaces $L_q(p;1,p)$ could be viewed as being principal $U(1)$-comodule algebras over the quantum teardrops, or weighted projective spaces, $\mathbb{WP}_q(1,p)$.
This type of structure it was also studied for $r=1$ in ~\cite{1401.6788}, and when $r$ divides $p$ in~\cite{1409.5335}.
This result was generalized to higher-dimensional quantum weighted projective spaces $\mathbb{P}_q(p,\underline{l})$, where $\underline{l}$ is now a vector, in~\cite[Proposition 7.1]{dandrea_landi_2014}, with the condition that the product of the components of $\underline{l}$ is a power of $p$.

The $L_q(p;1,p)$ and $L_q(p;1,r)$, where $r$ divides $p$, lens spaces do not fit into the framework described above, as we only study lens spaces were the coaction of the finite group action is free, which is not the case for lens spaces of the form $L_q(p,r)$ with $p,r$ not coprime, as can be deduced from the proof of Proposition~\ref{prop:freeness of coaction}.

In this section, we show that the algebras $\mathcal{O}(L_q(p,r))$ are principal $U(1)$-comodule algebras over the quantum teardrop $\mathbb{WP}_q(1,r)$ is true for general quantum lens spaces, if $q^2 \neq 1$.

The teardrop orbifold of Thurston which we denote by $\mathbb{WP}(r_1,r_2)$, can be defined as the quotient of $S^3:=\{ (z_1,z_2)\in \C^2: |z_1|^2 + |z_2|^2=1\}$ by the following (twisted) action of $S^1:= \{t\in \C: |r|^2 = 1\}$:
\[ t\cdot (z_1,z_2) = (t^{r_1} z_1, t^{r_2} z_2).\]
Of course, the teardrop $\mathbb{WP}(n,n)$ for any $n>0$ is homeomorphic to the sphere $S^2$, but the the quotients where $r_1\neq r_2$ are not manifolds anymore, but orbifolds (in fact, what is usually called a ``bad'' orbifold, meaning that it there doesn't exist a finite covering by a simply connected manifold).

The quantum teardrop can be defined similarly, as the subalgebra of $SU_q(2)$ invariant under a suitable action of $U(1)$:
\[ t \cdot (a,b) = (t^{r_1} a,t^{r_2} b),\]
with $a$ and $b$ the generators of the $SU_q(2)$ $C^*$-algebra. The invariant subalgebra is of course the algebra generated as a vector space by basis vectors of the form $e_{klm}$ of \eqref{eqn:suq2 space algebra basis}, such that $r_1 k + r_2(l-m) = 0$. 

It is not hard to see that this algebra is generated by the elements $b b^*$ and $a^{r_2} {(b^*)}^{r_1}$ of $\A(SU_q(2))$. In~\cite{MR2989456} it was also shown that on the $C^*$-algebra level, the quantum teardrops $\mathbb{WP}_q(1,r)$ and $\mathbb{WP}_q(r_1,r)$ are isomorphic.

In order to describe our results on fiber bundles, we will switch to coactions for this section. A continuous coaction $\rho$ for a Hopf algebra $H$ acting on a $C^*$-algebra $A$ is a map $\rho: A \rightarrow A\ts H$ that has the following properties:
\begin{itemize}
 \item $\rho$ is injective
 \item $\rho$ is a comodule structure: $(1\ts \Delta)\circ \rho = (\rho \ts 1)\circ \rho$, where $\Delta$ is the coproduct of $H$.
 \item Podle\'s condition: $\rho(\A)(1\ts H) = \A\ts H$.
\end{itemize}

This coaction can be used to define principal $H$-comodule algebras, which can be seen as a a generalization of the concept of a principal fiber bundle to noncommutative geometry~\cite{MR1461943},\cite{MR2038278},~\cite{MR1098988}.

Let $A$ be a $C^*$-algebra, with a coaction $\rho: A \rightarrow A\ts H$ by a Hopf algebra $H$. Denote by $B$ the coinvariant part of $A$, i.e.\ the part where $\rho(h)a = a\ts 1$.This is a \emph{quantum principal fibration}, or principal $H$-comodule algebra if:
\begin{itemize}
\item The canonical map $\can: A\ts_B A \rightarrow A\ts H: a\ts a' \mapsto a\rho(a')$ is a bijection.
\item The map $B\ts A\rightarrow A: b\ts a\mapsto ba$ splits as a left $B$-module and and a right $H$-comodule map. This is also called equivariant projectivity.
\end{itemize}

Because of the results of~\cite{dabrowksi_gosse_hajac_2001} and~\cite{MR2989456}, a right $H$-comodule algebra $A$ is principal if and only if there exists a strong connection, i.e.\ there exists a map $\omega: H\rightarrow A\ts A$ such that:
\begin{subequations}
\begin{align}
\omega(1) &= 1\ts 1\label{eqn:strong connection a}\\
\mu \circ \omega &= \eta \circ \epsilon\label{eqn:strong connection b}\\
(\omega \ts \text{id}) \circ \Delta &= (\text{id} \ts \rho)\circ \omega\label{eqn:strong connection c}\\
(S\ts \omega) \circ \Delta &= (\sigma \ts \text{id})\circ (\rho\ts \text{id})\circ \omega,\label{eqn:strong connection d}
\end{align}
\end{subequations}
where $S,\eta,\epsilon,\Delta$ are the antipode, unit, counit and comultiplication of the Hopf algebra $H$,  
$\sigma: A\ts H \rightarrow H\ts A$ is the flip and $\mu: a \ts a' \to a a'$ the product.

A strong connection in the case where $H=U(1)$ is particularly nice to work with, because of the following lemma:
\begin{lemma}\label{lem:U(1) strong connection}
An algebra $A$, with continuous coaction $\rho: A \rightarrow A\ts U(1)$, has a strong connection if and only if there exists elements $\sum_i a_i\ts b_i$, and $\sum_i b_i' \ts a_i'$ such that $\sum a_i b_i = \sum b_i' a_i'= 1$, 
for $a_i$ and $a_i'$ of degree $-1$ and $b_i$ and $b_i'$ of degree $+1$ for the coaction.\footnote{While in corrections, we were made aware of~\cite{1409.5335}, where the same lemma was proven using a different method.}
\end{lemma}
The proof of this lemma is a slight generalization of a construction done in the proof of~\cite[Theorem 3.3]{MR2989456}. 
It also follows from this that a principal $U(1)$-fiber bundle over $A_0$ is \emph{strongly $\Z$-graded}, i.e.\ there is a $\Z$-grading $A = \oplus_{k\in\Z} A_k$ such that $A_k A_l = A_{k+l}$. This immediately follows from the lemma above.
\begin{proof}
 We construct a strong connection by induction. Define $\omega$ by:
\begin{align*}
\omega(1) &= 1\ts 1\\
\omega(u^n) &= \sum_i a_i \omega(u^{n-1}) b_i\\
\omega(u^{-n}) &= \sum_i b_i' \omega(u^{-n+1}) a_i',
\end{align*}
for each $n\geq 1$. Condition~\eqref{eqn:strong connection a} is immediate. For $n=1$, we see that because $\sum a_i b_i = \sum b_i' a_i'= 1 = \eta(\epsilon(u))$, condition~\eqref{eqn:strong connection b} is satisfied. Conditions~\eqref{eqn:strong connection c} and~\eqref{eqn:strong connection d} are also obvious by the definition.

Now suppose $\omega(u^{n-1})$ satisfies conditions~\eqref{eqn:strong connection b}--\eqref{eqn:strong connection d}. Then 
\[\mu(\omega(u^n)) = \mu( \sum_i a_i \omega(u^{n-1}) b_i) = 1,\] because $\mu(\omega(u^{n-1}))=1$. We also see that 
\begin{align*}(\text{id}\ts \rho) (\omega(u^n))&= (\text{id}\ts \rho) \sum_i a_i \omega(u^{n-1}) b_i\\
&= \sum_i a_i \omega(u^{n-1}) b_i \ts u^n\\
&= \omega(u^n)\ts u^n.
\end{align*}
The same argument also works for condition~\eqref{eqn:strong connection d}, and the $u^{-n}$ case.

If the conditions of the lemma are not satisfied, it cannot be a fibration, because the map $\can$ is either missing $1\ts u$ or $1\ts u^{-1}$ from its image.
\end{proof}
It is immediate from the Lemma that the $a_i$ generate $A_1$ as a right module over $A_0$, and $a_i'$ generate $A_1$ as a left module over $A_0$, by observing that $\forall e\in A_1$ we have $e = \sum_i a_i b_i e$ and $b_i e \in A_0$.

\begin{theorem}\label{thm:lens spaces are fibrations, coordinate version}
The coordinate algebra $\mathcal{O}(L_q(p,r))$, with $r\neq 0$, is a quantum principal $U(1)$-fibration over the quantum teardrop $\mathcal{O}(\mathbb{WP}_q(1,r))$ when $0\leq q< 1$.\end{theorem}
Of course, for $q=1$, the circle action is not free if $r\neq 1$. 

\begin{proof}
For basis vector $e_{klm} \in L_q(p,r)$, we have $k + r(l-m) = np$ for some $n\in \Z$. Define the coaction $\rho$ of $U(1)$ to be $\rho(e_{klm}) = e_{klm} \ts u^n$, with $n$ the number above. Clearly $\mathbb{WP}_q(1,r)$ is the coinvariant subalgebra under this coaction.

We first consider the case $q\neq 0$. Observe that $a^p$ is an element of degree $+1$, and ${(a^*)}^p$ is of degree $-1$. Now set $m$ to be an integer such that $p - r m<0$, then $a^{-p+rm} {(b^*)}^m$ is also of degree $-1$.

We see that 
\[ a^p {(a^*)}^p = \prod_{j=0}^{p-1} (1- q^{-2j} bb^*), \]
and
\[b^m {(a^*)}^{-p+rm} a^{-p+rm}{(b^*)}^m = q^{m (-p+rm)}  \prod_{j=1}^{rm-p} (1- q^{2j} bb^*) {(bb^*)}^m.\] 
Setting $bb^* = x$, we see that the first expression is a polynomial in $x$ with roots $x=q^{2j}$ for $j=0,1,\ldots, p-1$ and 
the second is a polynomial in $x$ with roots $x=q^{-2j}$ with $j=1,2,\ldots, (rm-p)$ and $x=0$. 
By B\'ezout's identity, since the two polynomials have no common divisor we know there exist polynomials 
$g_1$, $g_2$ such that
\[ g_1(bb^*) a^p {(a^*)}^p + g_2(bb^*) b^m {(a^*)}^{-p+rm}  a^{-p+rm}{(b^*)}^m = 1.\]
Because polynomials in $bb^*$ lie in degree $0$ of the coaction, this defines a strong connection by Lemma~\ref{lem:U(1) strong connection}, with the left degree $+1$ and the right degree $-1$. Of course, we can do the 
same for the left-degree $-1$, and right degree $+1$.

If $q=0$, the left degree $-1$ part is easy, because we just need to observe that ${(a^*)}^p a^p =1$, however the left degree $+1$ part is somewhat harder, because the previous construction will involve powers $q^{-2k}$. However, it can be shown that
\begin{equation} \sum_{p_1 = 0}^p a^{p-p_1} b^{p_1} {(a^*)}^{(r-1)p_1} a^{(r-1)p_1} {(b^*)}^{p_1} {(a^*)}^{p-p_1} = 1,\label{eqn:decomposition of 0 when q 0}\end{equation}
which gives a correctly graded decomposition, with $a_i =  a^{p-p_1} b^{p_1} {(a^*)}^{(r-1)p_1}$, since $(p-p_1) + r p_1 - (r-1)p_1 = p$. 

To show that the sum in \eqref{eqn:decomposition of 0 when q 0} equals $1$, first observe that $a^* a =1$, hence we can reduce the sum to 
\[ \sum_{p_1 = 0}^p a^{p-p_1} b^{p_1} {(b^*)}^{p_1} {(a^*)}^{p-p_1}.\]
Then, observing that ${(bb^*)}^2 = b b^*$, this can be further reduced to
\[ a^p {(a^*)}^p + \sum_{p_1 = 1}^{p-1} a^{p-p_1} b b^* {(a^*)}^{p-p_1} + b b^*.\]
We can collapse it to something even simpler still, by observing that for all $k>0$, we have $a^k {(a^*)}^k = a^{k-1} (1 - b b^*) {(a^*)}^{k-1}$. Hence
\[  a^p {(a^*)}^p + \sum_{p_1 = 1}^{p-1} a^{p-p_1} b b^* {(a^*)}^{p-p_1} = a a^*,\]
and the sum reduces to $a a^* + b b^* =1$.

Thus we have proven that the conditions of Lemma~\ref{lem:U(1) strong connection} are satisfied, hence $\mathcal{O}(L_q(p,r))$ is a quantum principal fibration over $\mathcal{O}(\mathbb{WP}_q(1,r))$.
\end{proof}

In the future we hope to use the above characterization of quantum lens spaces as a total space for a principal $U(1)$-bundle over quantum teardrops to define real spectral triples on the base space. The classical teardrops are not manifolds and therefore the Dirac operator and the spectral triple for them make sense only when considered over the covering space. It remains open whether the orbifold-type singularities disappear when one considers $q$-deformed objects as some studies suggest (see~\cite{Brzezinski:12,MR2989456}). Indeed, a similar approach has been implemented in~\cite{Harju}, where an example of an odd spectral triple over quantum weighted projective spaces was constructed. However, when it comes to spectral triples we cannot, in contrast to~\cite{dabrowski_sitarz_2013, dabrowski_sitarz_zucca_2013},
claim that the spectral triple over the quantum lens space is projectable.

\bibliographystyle{plain}

\begin{bibdiv}
\begin{biblist}

\bib{1401.6788}{article}{
    author = {Arici, Francesca},
    author = {Brain, Simon},
    author = {Landi, Giovanni},
    title = {The {G}ysin Sequence for Quantum Lens Spaces},
    date = {2014},
    month = {Jan},
    archivePrefix = {arXiv},
}

\bib{1409.5335}{article}{
    author = {Arici, Francesca},
    author = {Kaad, Jens},
    author = {Landi, Giovanni},
    title = {Pimsner algebras and {G}ysin sequences from principal circle actions},
    date = {2014},
    month = {sep},
    archivePrefix = {arXiv},
}



\bib{bar_dirac_1992}{article}{
      author={B\"ar, Christian},
       title={The {D}irac operator on homogeneous spaces and its spectrum on
  3-dimensional lens spaces},
        date={1992},
        ISSN={{0003-889X}},
     journal={Archiv der Mathematik},
      volume={59},
      number={1},
       pages={65\ndash 79},
         url={http://www.springerlink.com/content/l8565611612716m7/},
}

\bib{Brzezinski:12}{article}{
      author={Brzezi{\'n}ski, Tomasz},
       title={On the smoothness of the noncommutative pillow and quantum
  teardrops},
        date={2013},
      eprint={1311.4758},
}

\bib{MR2989456}{article}{
      author={Brzezi{\'n}ski, Tomasz},
      author={Fairfax, Simon~A.},
       title={Quantum teardrops},
        date={2012},
        ISSN={0010-3616},
     journal={Comm. Math. Phys.},
      volume={316},
      number={1},
       pages={151\ndash 170},
         url={http://dx.doi.org/10.1007/s00220-012-1580-2},
      review={\MR{2989456}},
}

\bib{MR2038278}{article}{
      author={Brzezi{\'n}ski, Tomasz},
      author={Hajac, Piotr~M.},
       title={The {C}hern-{G}alois character},
        date={2004},
        ISSN={1631-073X},
     journal={C. R. Math. Acad. Sci. Paris},
      volume={338},
      number={2},
       pages={113\ndash 116},
         url={http://dx.doi.org/10.1016/j.crma.2003.11.009},
      review={\MR{2038278}},
}

\bib{connes_noncommutative_1995}{article}{
      author={Connes, Alain},
       title={Noncommutative geometry and reality},
        date={1995},
        ISSN={0022-2488},
     journal={Journal of Mathematical Physics},
      volume={36},
      number={11},
       pages={6194\ndash 6231},
         url={http://dx.doi.org/10.1063/1.531241},
      review={\MR{MR1355905 (96g:58014)}},
}

\bib{connes_gravity_1996}{article}{
      author={Connes, Alain},
       title={Gravity coupled with matter and the foundation of non-commutative
  geometry},
        date={1996},
        ISSN={0010-3616},
     journal={Communications in Mathematical Physics},
      volume={182},
      number={1},
       pages={155\ndash 176},
         url={http://projecteuclid.org/getRecord?id=euclid.cmp/1104288023},
      review={\MR{MR1441908 (98f:58024)}},
}

\bib{dabrowksi_gosse_hajac_2001}{article}{
      author={D{\polhk{a}}browski, Ludwik},
      author={Grosse, Harald},
      author={Hajac, Piotr~M.},
       title={Strong connections and {C}hern-{C}onnes pairing in the
  {H}opf-{G}alois theory},
        date={2001},
        ISSN={0010-3616},
     journal={Comm. Math. Phys.},
      volume={220},
      number={2},
       pages={301\ndash 331},
         url={http://dx.doi.org/10.1007/s002200100433},
      review={\MR{1844628 (2002g:58007)}},
}

\bib{dbrowski_dirac_2005}{article}{
      author={D{\c{a}}browski, Ludwik},
      author={Landi, Giovanni},
      author={Sitarz, Andrzej},
      author={van Suijlekom, Walter},
      author={V{\'a}rilly, Joseph~C.},
       title={The {D}irac operator on {${\mathrm{SU}}_q(2)$}},
        date={2005},
        ISSN={0010-3616},
     journal={Communications in Mathematical Physics},
      volume={259},
      number={3},
       pages={729\ndash 759},
         url={http://dx.doi.org/10.1007/s00220-005-1383-9},
      review={\MR{2174423 (2006h:58034)}},
}


\bib{dabrowski_sitarz_2013}{article}{
    AUTHOR = {D{\polhk{a}}browski, Ludwik},
    author= {Sitarz, Andrzej},
     TITLE = {Noncommutative circle bundles and new {D}irac operators},
     date = {2013},
     ISSN = {0010-3616},
   JOURNAL = {Comm. Math. Phys.},
    VOLUME = {318},
    NUMBER = {1},
     PAGES = {111--130},
     URL = {http://dx.doi.org/10.1007/s00220-012-1550-8},
  review = {\MR{3017065}},
}

\bib{dabrowski_sitarz_zucca_2013}{article}{
      author={D{\polhk{a}}browski, Ludwik},
      author={Sitarz, Andrzej},
      author={Zucca, Alessandro},
       title={Dirac operator on noncommutative principal circle bundles},
        date={2013},
      eprint={1305.6185},
}

\bib{dandrea_landi_2014}{article}{
      author={D'Andrea, Francesco},
      author={Landi, Giovanni},
       title={Quantum weighted projective and lens spaces},
        date={2014},
      eprint={1410.4508},
}


\bib{british1885report}{inproceedings}{
      author={Dyck, Walther},
       title={On the ``{A}nalysis situs'' of three-dimensional spaces},
        date={1885},
   booktitle={Report of the {F}ifty-fourth {M}eeting of the {B}ritish
  {A}ssociation for the {A}dvancement of {S}cience: Held at {M}ontreal in
  {A}ugust and {S}eptember 1884},
   publisher={J. Murray},
         url={http://books.google.nl/books?id=ZGsjQwAACAAJ},
}

\bib{Ellwood:2000}{article}{
 author ={Ellwood, David A.},
 title={A new characterisation of principal actions},
	date = {2000},
	ISSN = {0022-1236},
   JOURNAL = {J. Funct. Anal.},
    VOLUME = {173},
      YEAR = {2000},
    NUMBER = {1},
     PAGES = {49\ndash 60},
       URL = {http://dx.doi.org/10.1006/jfan.2000.3561},
  review = {\MR{1760277 (2001c:46126)}},
}

\bib{franc_spin_1987}{article}{
      author={Franc, Annick},
       title={Spin structures and {K}illing spinors on lens spaces},
        date={1987},
        ISSN={0393-0440},
     journal={J. Geom. Phys.},
      volume={4},
      number={3},
       pages={277\ndash 287},
         url={http://dx.doi.org/10.1016/0393-0440(87)90015-5},
      review={\MR{957015 (90e:57047)}},
}

\bib{bonda_elements_2001}{book}{
      author={Gracia-Bond{\'{\i}}a, Jos{\'e}~M.},
      author={V{\'a}rilly, Joseph~C.},
      author={Figueroa, H{\'e}ctor},
       title={Elements of noncommutative geometry},
      series={Birkh\"auser Advanced Texts: Basel Textbooks},
   publisher={Birkh\"auser Boston Inc.},
     address={Boston, MA},
        date={2001},
        ISBN={0-8176-4124-6},
      review={\MR{MR1789831 (2001h:58038)}},
}

\bib{MR2242563}{article}{
      author={Hadfield, Tom},
      author={Kr{\"a}hmer, Ulrich},
       title={Twisted homology of quantum {${\mathrm{SL}}(2)$}},
        date={2005},
        ISSN={0920-3036},
     journal={$K$-Theory},
      volume={34},
      number={4},
       pages={327\ndash 360},
         url={http://dx.doi.org/10.1007/s10977-005-3118-2},
      review={\MR{2242563 (2007j:58009)}},
}


\bib{MR1461943}{article}{
      author={Hajac, Piotr~M.},
       title={Strong connections on quantum principal bundles},
        date={1996},
        ISSN={0010-3616},
     journal={Comm. Math. Phys.},
      volume={182},
      number={3},
       pages={579\ndash 617},
         url={http://projecteuclid.org/getRecord?id=euclid.cmp/1104288302},
      review={\MR{1461943 (98e:58022)}},
}

\bib{Harju}{article}{
   author = {{Harju}, A.~J.},
    title = {Dirac Operators on Quantum Weighted Projective Spaces},
  journal = {ArXiv e-prints},
archivePrefix = {arXiv},
   eprint = {1402.6251},
 primaryClass = {math.QA},
 keywords = {Mathematics - Quantum Algebra, Mathematical Physics},
     year = {2014},
    month = {feb},
   adsurl = {http://adsabs.harvard.edu/abs/2014arXiv1402.6251H},
  adsnote = {Provided by the SAO/NASA Astrophysics Data System}
}


\bib{MR2015735}{article}{
      author={Hong, Jeong~Hee},
      author={Szyma{\'n}ski, Wojciech},
       title={Quantum lens spaces and graph algebras},
        date={2003},
        ISSN={0030-8730},
     journal={Pacific Journal of Mathematics},
      volume={211},
      number={2},
       pages={249\ndash 263},
         url={http://dx.doi.org/10.2140/pjm.2003.211.249},
      review={\MR{2015735 (2004g:46074)}},
}

\bib{ikeda_spectra_1979}{article}{
      author={Ikeda, Akira},
      author={Yamamoto, Yoshihiko},
       title={On the spectra of {$3$}-dimensional lens spaces},
        date={1979},
        ISSN={0030-6126},
     journal={Osaka J. Math.},
      volume={16},
      number={2},
       pages={447\ndash 469},
         url={http://projecteuclid.org/getRecord?id=euclid.ojm/1200772088},
      review={\MR{539600 (80e:58042)}},
}

\bib{iochum_classification_2004}{article}{
      author={Iochum, Bruno},
      author={Sch{\"u}cker, Thomas},
      author={Stephan, Christoph},
       title={On a classification of irreducible almost commutative
  geometries},
        date={2004},
        ISSN={0022-2488},
     journal={J. Math. Phys.},
      volume={45},
      number={12},
       pages={5003\ndash 5041},
         url={http://dx.doi.org/10.1063/1.1811372},
      review={\MR{2105233 (2005j:58038)}},
}


\bib{OlczykowskiSitarz:2013}{article}{
      author={Olczykowski, Piotr},
      author={Sitarz, Andrzej},
       title={Real spectral triples over noncommutative {B}ieberbach
  manifolds},
        date={2013},
        ISSN={0393-0440},
     journal={J. Geom. Phys.},
      volume={73},
       pages={91\ndash 103},
         url={http://dx.doi.org/10.1016/j.geomphys.2013.05.003},
      review={\MR{3090104}},
}

\bib{reidemeister_homotopieringe_1935}{article}{
      author={Reidemeister, Kurt},
       title={Homotopieringe und {L}insenr\"aume},
        date={1935},
        ISSN={0025-5858},
     journal={Abh. Math. Sem. Univ. Hamburg},
      volume={11},
      number={1},
       pages={102\ndash 109},
         url={http://dx.doi.org/10.1007/BF02940717},
      review={\MR{3069647}},
}

\bib{MR1098988}{article}{
      author={Schneider, Hans-J{\"u}rgen},
       title={Principal homogeneous spaces for arbitrary {H}opf algebras},
        date={1990},
        ISSN={0021-2172},
     journal={Israel J. Math.},
      volume={72},
      number={1-2},
       pages={167\ndash 195},
         url={http://dx.doi.org/10.1007/BF02764619},
        note={Hopf algebras},
      review={\MR{1098988 (92a:16047)}},
}



\bib{MR3108692}{article}{,
    AUTHOR = {Teh, Kevin},
     TITLE = {Nonperturbative spectral action of round coset spaces of {SU}(2)},
   JOURNAL = {J. Noncommut. Geom.},
    VOLUME = {7},
      YEAR = {2013},
    NUMBER = {3},
     PAGES = {677\ndash 708},
      ISSN = {1661-6952},
  review = {\MR{3108692}},
       URL = {http://dx.doi.org/10.4171/JNCG/131},
}

\bib{MR890482}{article}{
      author={Woronowicz, Stanis{\l}aw~L.},
       title={Twisted {${\mathrm SU}(2)$} group. {A}n example of a
  noncommutative differential calculus},
        date={1987},
        ISSN={0034-5318},
     journal={Kyoto University. Research Institute for Mathematical Sciences.
  Publications},
      volume={23},
      number={1},
       pages={117\ndash 181},
         url={http://dx.doi.org/10.2977/prims/1195176848},
      review={\MR{890482 (88h:46130)}},
}

\bib{MR558314}{article}{
      author={Yamamoto, Yoshihiko},
       title={On the number of lattice points in the square {$x+y\leq u$} with
  a certain congruence condition},
        date={1980},
        ISSN={0030-6126},
     journal={Osaka J. Math.},
      volume={17},
      number={1},
       pages={9\ndash 21},
         url={http://projecteuclid.org/getRecord?id=euclid.ojm/1200772803},
      review={\MR{558314 (81c:10062)}},
}

\end{biblist}
\end{bibdiv}
\end{document}